\documentclass[12pt, one side,emlines]{amsart}
\usepackage{amssymb,latexsym,xy,eucal,mathrsfs, graphicx, tikz}
\textwidth=15.2cm \textheight=20cm \theoremstyle{plain}
\newtheorem{lemma}{Lemma}[section]
\newtheorem{proposition}[lemma]{Proposition}
\newtheorem{corollary}[lemma]{Corollary}
\usepackage{graphicx}
\theoremstyle{definition}
\newtheorem{example}[lemma]{Example}

\newtheorem{definition}[lemma]{Definition}

\usepackage[bottom=0.5in,top=1.5in]{geometry}
\numberwithin{equation}{section}  \voffset
-55truept \hoffset -15truept
\begin{document}
	\baselineskip 15truept
\title{On spectrum of the zero-divisor graph of matrix ring}
\maketitle 
\begin{center}
	  Krishnat Masalkar, Anil Khairnar,  Anita Lande, Lata Kadam\\
	
\end{center} 

\begin{abstract}
	For a ring $R$, the zero-divisor graph is a simple graph $\Gamma(R)$ whose vertex set is the set of all non-zero zero-divisors in a ring $R$, and two distinct vertices $x$ and $y$ are adjacent if and only if $xy=0$ or $yx=0$ in $R$. By using Weyl's inequality we give bounds on eigenvalues of adjacency matrix of  $\Gamma(M_2(F))$, where $M_2(F)$ is a $2 \times 2$ matrix ring over a  finite field $F$. 
\end{abstract}
\section{Introduction}\label{sec1}

The concept of the zero-divisor graph of a commutative ring was first introduced by I. Beck \cite{beck1988coloring} in $1988$.
Being motivated by Beck, in \cite{anderson1999zero} Anderson and Livingston  defined a zero-divisor graph for commutative rings.
They defined the {\it zero-divisor graph} for a commutative ring $R$ as a simple (undirected) graph, whose vertices are the nonzero zero-divisors of $R$ and two distinct vertices $x$ and $y$ are adjacent if and only if $xy=0$.  Redmond \cite{redmond2002zero} defined the zero-divisor graph for non-commutative ring $R$  denoted by $\Gamma(R)$ to be a simple (undirected) graph, whose vertices are the nonzero zero-divisors of $R$ and two distinct vertices $x$ and $y$ are adjacent if and only if $xy=0$ or $yx=0$.
In \cite{khairnar2016zero}, authors studied the diameter and girth of the zero-divisor graph under extension to Laurent polynomial and Laurent power series rings.
\par The adjacency matrix $A(G)=[a_{ij}]_{n\times n}$ of a graph $G$ with $n$ vertices is the matrix  such that, $a_{ij} = 1$ if $i-j\in E(G)$ and $a_{ij}= 0$ otherwise. Laplacian matrix of the graph $G$ is given by $L(G)=diag(d(1),\cdots,d(n))-A(G).$ A multiset of eigenvalues $\sigma_A(G) =\left\{\lambda_1^{(s_1)}, \cdots,\lambda_n^{(s_n)}\right\}$ of $A(G)$ is called as the adjacency spectrum of $G.$  The Laplacian spectrum $\sigma_L(G)$ of a graph $G$ is defined as the multiset of eigenvalues of $L(G).$  The author's refer  \cite{godsil2001algebraic} for concepts in graph theory and spectral graph theory. 
The spectra of zero-divisor graphs is studied in (\cite{patil2021spectrum}, \cite{pirzada2021eigenvalues},  \cite{chattopadhyay2020laplacian}, \cite{magi2020adjacency}, \cite{Jitsupat2021eigenvalues},\cite{monius2021eigenvalues}).
In \cite{cardoso2013spectra}, Domingos M. Cardoso et.al studied the adjacency and Laplacian spectra of graphs obtained by generalized join graph operation on a family of graphs. 
Throughout this paper $F$ denotes a finite field of order $n+1$ and $Z(R)$ denotes the set of nonzero zero-divisors in a ring $R$. 
\par  In the second section,
we identify idempotent and nilpotent elements in $M_2(F)$. By classifying the idempotent and nilpotent elements in $M_2(F)$, we describe the spectrum of $\Gamma(M_2(F))$. We define a relation on a set of zero-divisors of $M_2(F)$, which is an equivalence relation. Also, we express the adjacency spectrum of   $\Gamma(M_2(F))$ in terms of the spectrum of $A(H)$, where $A(H)$ is the adjacency matrix of equivalence classes of idempotent and nilpotent elements in $Z(M_2(F))$. In the third section, using Weyl's inequality we give bounds for eigenvalues of adjacency matrix of $\Gamma(M_2(F))$.

\section{Idempotents and Nilpotents in $Z(M_2(F))$}
In this section, we identify idempotent and nilpotent elements in  $Z(M_2(F))$ and we classify them.\\

We use the following notations:
\begin{align*}
	& E_0=\begin{bmatrix} 0& 0\\0& 1\end{bmatrix},~E^0=\begin{bmatrix} 1& 0\\0& 0\end{bmatrix},~E_a=\begin{bmatrix} 0& 0\\a& 1\end{bmatrix},~
	E^a=\begin{bmatrix} 1& a\\0& 0\end{bmatrix},F^a=\begin{bmatrix} 0& a\\0& 1\end{bmatrix},\\&F_a=\begin{bmatrix} 1& 0\\a& 0\end{bmatrix},~
	E_{ij}=\begin{bmatrix} i& j(1-i)\\ \frac{i}{j}& 1-i\end{bmatrix},~
	N=\begin{bmatrix} 0& 1\\0& 0\end{bmatrix},~M=\begin{bmatrix} 0& 0\\1& 0\end{bmatrix},~
	N_k=\begin{bmatrix} 1& k\\ -\frac{1}{k}& -1\end{bmatrix}.
\end{align*}
where  $i,~j,~a,~k \in F$.\\

In the following lemma, we identify all idempotent elements in $Z(M_2(F))$.
\begin{lemma}\label{lemma2.1}
	Every idempotent element in $Z(M_2(F))$ is one of the following form:\\
	$(i)$ \begin{align*}
		E_0,E^0, E_a,~
		E^a,F_a,~F^a~~\text{for some}~~a\in F\backslash\{0\},
	\end{align*}
	$(ii)$ $$E_{ij}~~\text{for some nonzero }~~ i\in  F\backslash\{0,1\},~~j\in F\backslash\{0\}.$$
\end{lemma}
\begin{proof}
	Let $A=\begin{bmatrix} a& b\\c& d\end{bmatrix}$ be an idempotent in $Z(M_2(F)).$ Since $A$ is nonzero non-invertible idempotent matrix of size 2, minimal polynomial of $A$ is $x^2-x=x^2-(a+d)x+ad-bc.$ Therefore $d=1-a$ and $bc=a(1-a).$ If $a=0~or~1$ then 
	$A$ has one of the form $$\begin{bmatrix} 0& 0\\0& 1\end{bmatrix},~\begin{bmatrix} 1& 0\\0& 0\end{bmatrix},~\begin{bmatrix} 1& 0\\c& 0\end{bmatrix},~
	\begin{bmatrix} 1& b\\0& 0\end{bmatrix},~\begin{bmatrix} 0& b\\0& 1\end{bmatrix},~\begin{bmatrix} 0& 0\\c& 1\end{bmatrix}.$$
	If $a\neq 0$ and $a\neq 1$ then $A$ has  the form 
	$$\begin{bmatrix} a& k(1-a)\\ \frac{a}{k}& 1-a\end{bmatrix}~~\text{for some }~~ k\in  F\backslash \{0\}.$$
	
	If we put $a=0$ or $a=1$ in $\begin{bmatrix} a& k(1-a)\\ \frac{a}{k}& 1-a\end{bmatrix}$  then we get
	$$\begin{bmatrix} a& k(1-a)\\ \frac{a}{k}& 1-a\end{bmatrix}=\begin{bmatrix} 0& k\\0& 1\end{bmatrix}~~or~~\begin{bmatrix} a& k(1-a)\\ 
		\frac{a}{k}& 1-a\end{bmatrix}=\begin{bmatrix} 1& 0\\1/k& 0\end{bmatrix}$$ for $k\neq 0$, which is of the form  $\begin{bmatrix} 1& b\\0& 0\end{bmatrix}$ or $\begin{bmatrix} 1& 0\\c& 0\end{bmatrix}.$
\end{proof}

In the following lemma, we identify nilpotent elements in $Z(M_2(F))$. 
\begin{lemma}\label{lemma2.3}
	Every nilpotent element in $Z(M_2(F))$ is one of the following form\\
$	(i)$ $$aN,~~aM~~~\text{for some}~~a\in F\backslash\{0\},$$
	$(ii)$ $$aN_k~~\text{for some }~ a, k\in  F\backslash\{0\}.$$
\end{lemma}
\begin{proof}
	Let $A=\begin{bmatrix} a& b\\c& d\end{bmatrix}$ be a nilpotent element in $Z(M_2(F))$. Since $A$ is a nonzero nilpotent matrix of size 2, the minimal polynomial of $A$ is $x^2=x^2-(a+d)x+ad-bc.$ Therefore $a=-d$ and $bc=-a^2.$ If $a=0$ then 
	$A$ has one of the form $$\begin{bmatrix} 0& b\\0& 0\end{bmatrix},~\begin{bmatrix} 0& 0\\c& 0\end{bmatrix}.$$
	If $a\neq 0$ then $A$ has  the form 
	$$\begin{bmatrix} a& ka\\ \frac{-a}{k}& -a\end{bmatrix}~~\text{for some }~~ k\in  F\backslash \{0\}.$$ 
\end{proof}
We define a relation $\sim$ on  $Z(M_2(F))$, by $A\sim B$ in $Z(M_2(F))$ if and only if $A=UB=BV$ for some $U, V\in GL_2(F)$. Note that the relation $\sim$ is an equivalence relation.\\

In the following lemma, we determine equivalence classes of the relation $\sim$.
\begin{lemma}\label{lemma2.5}
	Equivalence classes of the relation $\sim$ are
	$$\left\{ \left[ E\right]  ,\left[N\right], \left[M\right],\left[N_k \right] ~:
	~E^2=E \in Z(M_2(F)), k\in F\setminus\{0\}  \right\}.$$
\end{lemma}
\begin{proof} First we show that every element of $Z(M_2(F))$ is  related to $E$ or $N$ or $M$ or $N_k$. 
	Let $B=\begin{bmatrix}x& y\\z& w \end{bmatrix}\in Z(M_2(F)).$ Then $rank(B)=1,$ so the
	second row is scalar multiple of the first row. Therefore  $B=\begin{bmatrix} x& y\\ kx & ky\end{bmatrix}$ with $x\neq0$ or $y\neq 0$ or $B=\begin{bmatrix} 0& 0\\ z & w\end{bmatrix}$.\\
	Suppose $B=\begin{bmatrix}x& y\\kx& ky \end{bmatrix}$ is  not nilpotent, therefore $x+ky\neq 0$.
	Let $E=\frac{1}{x+ky}B$, then $E^2=E, ~EB=BE=B.$
	If $U=(x+ky)\begin{bmatrix} 1& 0\\ 0& 1\end{bmatrix}$, then $U\in GL_2(F)$ and  $B=EU=UE.$ Therefore for each non-nilpotent element $B\in Z(M_2(F))$ there exists an idempotent $E\in Z(M_2(F))$ such that $B\sim E.$ Moreover, if we have another idempotent $F$ such that $B\sim F$ then $E\sim F$ and hence $EF=FE=E=F.$ Therefore, there is an unique idempotent $e_B$ such that $B\sim e_B.$\\
	If $B=\begin{bmatrix} 0& 0\\ z & w\end{bmatrix}$ is not nilpotent then $w\neq 0.$ Let $G=\begin{bmatrix} 0& 0\\ z/w & 1\end{bmatrix}$ and $V=w\begin{bmatrix} 1& 0\\ 0 & 1\end{bmatrix}.$ Then $G^2=G,$ $B=VG=GV$ and $V\in GL_2(F).$ Hence there exists an unique idempotent $e_B=G$ such that $B\sim e_B.$\\
	If $B=\begin{bmatrix} x& y\\kx& ky\end{bmatrix}$ is nilpotent.
	Then $x+ky=0$ and $B=\begin{bmatrix} x& y\\kx& -x\end{bmatrix}.$\\
	If $x\neq 0$ then $BU=UB=N,$ where $N=\begin{bmatrix}1& y/x\\-x/y &-1\end{bmatrix}$ is nilpotent and\\ $U=\frac{1}{x}\begin{bmatrix}1&0\\0&1\end{bmatrix}\in GL_2(F).$ Therefore, $B\sim N$ and $N$ is nilpotent.\\ 
	If $x=0$ then $ky=0.$ Since $B$ is nonzero, $y\neq 0.$ Let
	$V=y\begin{bmatrix}1&0\\0&1\end{bmatrix}$  and $M=\begin{bmatrix} 0& 1\\0&0\end{bmatrix}.$ Hence $B=VM=MV$ and $V\in GL_2(F)$. Therefore $B\sim M$ and $M$ is nilpotent. \\
	Let $B=\begin{bmatrix} 0& 0\\z& w\end{bmatrix}$. Then $w=0$ and $z\neq 0.$  Let $L=\begin{bmatrix} 0& 0\\1& 0\end{bmatrix}$ and $W=z\begin{bmatrix}1&0\\0&1\end{bmatrix}$.
	Therefore $B=WL=LW,$ $W\in GL_2(F)$ and $L$ is nilpotent.  Hence $B\sim L.$ 
	Therefore any $B\in Z(M_2(F))$ is in at least one of the equivalence classes $[E], [N], [M], [N_k]$.\\
	
	We show that no two different idempotents are related to each other under the relation $\sim$.
	Let $C, D\in M_2(F)$.
	Suppose $C^2=C\neq 0,~D^2=D\neq 0$. 
	If $C\sim D$ then 
	there exist units $U$ and $V$ such that $C=UD=DV.$  Hence $CD=UD^2=UD=C,~~DC=D^2V=DV=C.$ Therefore $CD=DC=C.$  Similarly, we can show that $CD=DC=D.$ So $C=D.$ Hence,  no two distinct idempotents are related under the relation $\sim.$\\
	We prove that $N, M, N_k$ are not related to each other under the relation $\sim$.
	If $C^2=D^2=0$ and  $C\sim D$ then there exist  units $U$ and $V$ such that $C=UD=DV$. Hence $DC=D^2V=0V=0$ and $CD=UD^2=U0=0$. Since 
	$NM, NN_k, MN_k, N_kN_j$ all are non-zero for $k\neq j$. 
	Therefore no two elements from $N, M, N_k$ are related under $\sim$.\\
	Now we prove that any idempotent element in $Z(M_2(F))$ is not related to any nilpotent element in $Z(M_2(F))$. 
	Let $C^2=C$ and $D^2=0.$ 
	If $C\sim D$ then there exist units $U$ and $V$ such that $C=UD=DV.$ Therefore, $C=C^2=UDDV=U0V=0$, a contradiction. 
\end{proof}
In the following lemma, we show that the cardinality of each equivalence class of the relation $\sim$ is the same and it is equal to $|F|-1$.
\begin{lemma}\label{2.6}
	If $[E]$ is an equivalence class  of the relation $\sim$ on $Z(M_2(F))$. Then
	$[E]=\left\{ aE ~:~ a\in F\setminus\{0\}\right\}$.
\end{lemma}
\begin{proof}
	Let $\begin{bmatrix} a&b\\c&d\end{bmatrix}\in \left[E_0\right]$ then there exist $B, C\in GL_2(F)$ such that \\$\begin{bmatrix} a&b\\c&d\end{bmatrix}=BE_0=E_0C$. This gives $\begin{bmatrix} a&b\\c&d\end{bmatrix}=\begin{bmatrix} 0&0\\0&d\end{bmatrix}$ with $d\in F\setminus\{0\}. $ \\
	
	Therefore   $~~\left[ E_0\right]=\left\{dE_0 ~:~ d\in F\setminus \{0\}\right\}.$\\ 		

	Similarly,\\	
	
	$\left[E^0 \right]=\left\{bE^0 ~:~ b\in F\setminus \{0\}\right\},
	~~~	\left[E_a \right] =\left\{ dE_a ~:~d\in F\setminus\{0\}\right\},$\\
	
		$	\left[E^b \right] =\left\{ cE^b ~:~c\in F\setminus\{0\}\right\},
	~~~	\left[F_c \right] =\left\{ aF_c ~:~a\in F\setminus\{0\}\right\},$\\
	
			$	\left[F^d \right] =\left\{ cF^d ~:~c\in F\setminus\{0\}\right\}.$\\

	Let $\begin{bmatrix} a&b\\c&d\end{bmatrix}\in\left[ E_{jk}\right].$ Then there exist invertible matrices $B$ and $C$ such that \\
	$\begin{bmatrix} a&b\\c&d\end{bmatrix}=BE_{jk}=E_{jk}C$.\\
	Applying row operation $R_1\longrightarrow R_2-\frac{1}{k}R_1$, we get that
	$\begin{bmatrix} a&b\\c-a/k&d-b/k\end{bmatrix}
	=\begin{bmatrix}
		j&k(i-j)\\0&0
	\end{bmatrix}C.$\\
	This imply $c=a/k, d=b/k$. Similarly, we get $b=k(1-a), d=1-kc.$\\
	Therefore $\begin{bmatrix} a&b\\c&d\end{bmatrix}
	=\begin{bmatrix}
		a& k(1-a)\\a/k & 1-a/k
	\end{bmatrix}=aE_{jk}$.\\

	Let $\begin{bmatrix}a&b\\c&d \end{bmatrix}\in [N_p]$. Then there exist invertible matrices $B$ and $C$ such that\\
	
	$\begin{bmatrix}a&b\\c&d \end{bmatrix}=B\begin{bmatrix}1&p\\-1/p&-1 \end{bmatrix}=\begin{bmatrix}1&p\\-1/p&-1 \end{bmatrix}C$.\\
	
	 Hence $\begin{bmatrix}a&b-pa\\c&d-cp \end{bmatrix}=B\begin{bmatrix}1&0\\-1/p&0 \end{bmatrix}.$ This imply $b=pa, d=cp.$\\ Also, we have \\$\begin{bmatrix}a&b\\c+a/p&d+b/p \end{bmatrix}=\begin{bmatrix}1&p\\0&0 \end{bmatrix}C.$ Therefore $c=-a/p, d=-b/p=-ap/p=-a.$ Hence $\begin{bmatrix}a&b\\c&d\end{bmatrix}=aN_p.$\\
	Similarly, $[N]=\left\{aN ~:~ a\in F\backslash\{0\} \right\}$ and $[M]=\left\{aM ~:~ a\in F\backslash\{0\} \right\}$.\\
	Hence for each $E\in \left\{M,N, N_p, E_0, E^0, E_a, E^b, F_c, F^d, E_{jk}\right\}$, we have\\
	$[E]=\left\{aE \colon a\in F\backslash\{0\}\right\}$ and $|[E]|=n$.
\end{proof}
As an application of the above equivalence relation $\sim$, we can count the number of elements in $Z(M_2(F))$ in two different ways as follows.\\
Let $F$ be finite field with $n=|F|-1$.
Then $|Z(M_2(F))|=|M_2(F)|-|GL_2(F)|-1=|F|^4-(|F|^2-1)(|F|^2-|F|)-1=(|F|-1)(|F|+1)^2=n(n+2)^2.$
Also the number of equivalence classes of the relation $\sim$  is
$4+4n+n(n-1)+n=(n+2)^2 $ and each equivalence class has $n$ elements. Hence $|Z(M_2(F)|=n(n+2)^2$. \\

	Let $H$ be a graph with the vertex set containing  all idempotent and all nilpotent elements in $Z(M_2(F))$, two vertices $x,y$ are adjacent in $H$ if and only if $xy=0$ or $yx=0.$\\
 
Let $F=\{a_0=0,a_1=1,a_2,a_3,a_4,\cdots,a_n\}$ be a finite field and  $ m=n(n-1)$.\\ Note the followings:\\
$F_{a_j}E_{a_k}=F^{a_j}E^{a_k}=0$,\\
$ E_{a_j}F^{a_k}=E^{a_j}F_{a_k}=0~\text{if and only if}~a_k=-\frac{1}{a_j}$,\\
$E_{a_j}N_{a_k}=0~~\text{if and only if}~a_k=\frac{1}{a_j}$,\\
$E^{a_j}N_{a_k}=0~~\text{if and only if}~a_k=a_j$,\\
$N_{a_k}F_{a_j}=0~~\text{if and only if}~a_k=-\frac{1}{a_j}$,\\
$N_{a_k}F^{a_j}=0~~\text{if and only if}~a_k=-a_j$,\\
$E_{a_j}E_{a_i,a_k}=0~~\text{if and only if}~a_k=-\frac{1}{a_j}$,\\
$E^{a_j}E_{a_i,a_k}=0~~\text{if and only if}~a_k=-a_j$,\\
$E_{a_i, a_k}F_{a_j}=0~~\text{if and only if}~a_i=\frac{a_k}{a_k-1/a_j}$,\\
$E_{a_i, a_k}F^{a_j}=0~~\text{if and only if}~a_i=\frac{a_k}{a_k-a_j}$,\\
$E_{a_i, a_k}N_{a_j}=0~~\text{if and only if}~a_i=\frac{a_k}{a_k+a_j}$,\\
$N_{a_j}E_{a_i, a_k}=0~\text{if and only if}~a_k=-a_j$,\\
$E_{a_i,a_j}E_{a_l,a_k}=0~\text{if and only if}~a_i=\frac{a_j}{a_j-a_k}$,\\
$E_0E^{a_j}=F_{a_j}E_0=E^0E_{a_j}=F^{a_j}E^0=ME_{a_j}=F_{a_j}M=NE^{a_j}=F^{a_j}N=E_0E^0=0$.\\

In the following result, we prove that the above graph $H$ is regular.
\begin{lemma}
	The graph $H$ is $(2n+3)$-regular.
\end{lemma}
\begin{proof}
	Let $\mathcal{N}(x)=\left\{ y\in H\backslash\{x\} \colon xy=0~\text{or}~yx=0\right\}$, be the neighborhood of $x$ in $H$.\\
	Observe that, \\
	$\mathcal{N}(E_0)=\left\{M,N,E^0,E^{a_j}, F_{a_j}\colon j=1,2,\cdots,n\right\},~\text{therefore}~|\mathcal{N}(E_0)|=2n+3.$\\
	$\mathcal{N}(E^0)=\left\{M,N,E_0,E_{a_j}, F^{a_j}\colon j=1,2,\cdots,n\right\},~\text{therefore}~|\mathcal{N}(E^0)|=2n+3.$\\
	$\mathcal{N}(M)=\left\{M,E_0,E^0,E_{a_j}, F_{a_j}\colon j=1,2,\cdots,n\right\},~\text{therefore}~|\mathcal{N}(M)|=2n+3.$\\
	$\mathcal{N}(N)=\left\{N,E_0,E^0,E^{a_j}, F^{a_j}\colon j=1,2,\cdots,n\right\},~\text{therefore}~|\mathcal{N}(N)|=2n+3.$\\
	$\mathcal{N}(E_{-1/a_j})=\left\{M,E^0,F^{a_j},F_{a_k},N_{a_j},E_{a_l,a_j}\colon l=2,\cdots,n; k=1,2,\cdots,n\right\}$, \\
	$\text{therefore}~|\mathcal{N}(E_{-1/a_j})|=2n+3$.\\
	$\mathcal{N}(E^{-a_j})=\{N,E_0,F^{a_k},F_{1/a_j},N_{a_j},E_{a_l,a_j}\colon l=2,\cdots,n; k=1,2,\cdots, n\},$\\ $\text{therefore}~|\mathcal{N}(E^{-a_j})|=2n+3.$\\
	$\mathcal{N}(F^{a_j})=\{M,E_0,E_{-1/a_j},E^{a_l},N_{a_j},E_{\frac{a_k}{a_k-a_j},a_k}\colon l=1,2,\cdots,n; k\neq 0,j\},$ \\ $\text{therefore}~|\mathcal{N}(F^{a_j})|=2n+3.$\\
	$\mathcal{N}(F_{1/a_j})=\{M,E_0,E_{a_l},E^{-a_j},N_{a_j},E_{\frac{a_k}{a_k-a_j},a_k}\colon l=1,2,\cdots,n; k\neq 0,j\},$\\ $\text{therefore}~|\mathcal{N}(F_{1/a_j})|=2n+3.$
	\\
	$\mathcal{N}(N_{a_j})=\{ E_{-1/a_j},E^{-a_j},F^{a_j},F_{1/a_j},N_{a_j},E_{a_i,a_j},E_{\frac{a_k}{a_k-a_j},a_k} \colon i=2,3,\cdots,n; k\neq 0,j\},\\~\text{therefore}~|\mathcal{N}(N_{a_j})|=2n+3.$\\
	$\mathcal{N}\left(E_{\frac{a_j}{a_j-a_k},a_j}\right)=\{E_{-1/a_j},E^{-a_j},F^{a_k},F_{1/a_k},N_{a_j},N_{a_k}, E_{a_l,a_k},E_{a_j,a_l} \colon \\ l=2,3,\cdots,n; k\neq 0,j\},\\~\text{therefore}~|\mathcal{N}(E_{\frac{a_j}{a_j-a_k},a_j})|=2n+3$.\\
	Hence $|\mathcal{N}(x)|=2n+3$ for any $x\in H$.
\end{proof}
Now we consider various subgraphs of the graph $H$ and find their spectra. \\
\noindent Let\\
$S_0=\left\{ M, N,E_0, E^0\right\},\\
S_{j}=\displaystyle \left\{ E_{\frac{-1}{a_j}},E^{-a_j}, F^{a_j},F_{\frac{1}{a_j}},N_{a_j}\right\},\\
T_j=\left\{E_{\frac{a_j}{a_j-a_i},a_j}~\colon~ i\neq 0,~i\neq j~\text{and for all}~ i=1,2,3, \cdots,n\right\}$.\\
First we consider the induced subgraph $H_1$ of $H$ with vertex set \\ $\displaystyle \bigcup_{j=1}^{n}S_j\setminus\left\{N_{a_j} \colon j=1,2,\cdots,n\right\}$\\
Note that,\\
$E_{-1/a_j}E_{a_i,a_k}=E^{-a_j}E_{a_i,a_k}=N_{a_j}E_{a_i,a_k}=0~\text{if and only if}~a_k=a_j,$\\ and
$E_{a_i,a_k}F_{1/a_j}=E_{a_i,a_k}F^{a_j}=E_{a_i,a_k}N_{a_j}=E_{a_i,a_k}E_{a_i,a_j}\textbf{}=0~\text{if and only if}~a_i=\frac{a_k}{a_k-a_j}$.\\
Let  $$C=\begin{bmatrix}0&0&0&1\\0&0&1&0\\0&1&0&0\\1&0&0&0\end{bmatrix},D=\begin{bmatrix}0&0&1&1\\0&0&1&1\\1&1&0&0\\1&1&0&0\end{bmatrix}.$$
The adjacency matrix of $H_1$ is given by  $$\begin{bmatrix} D&C&C&\dots&C\\C&D&C&\dots&C\\\vdots&\vdots&\vdots&\vdots&\vdots\\C&C&C&\dots&D
\end{bmatrix}.$$

In the following lemma, we find the determinant of a special type of block matrix.
\begin{lemma}\label{lem1.5}
	Let $B$ and $C$ be square matrices of the same size  and $A$ be a $n\times n$ block matrix, 
	$$A=\begin{bmatrix} C&B&B&\dots&B\\B&C&B&\dots&B\\\vdots&\vdots&\vdots&\vdots&\vdots\\B&B&B&\dots&C
	\end{bmatrix}.$$
	Then  $$\det(A)=\det(C+(n-1)B)\det(C-B)^{n-1}.$$
\end{lemma}
\begin{proof} Apply column transformations $\displaystyle C_1\rightarrow C_1+\sum_{i=2}^{n}C_j$ on A, then
	\begin{align*}
		\det(A)&=\det\begin{bmatrix} C+(n-1)B&B&B&\dots&B\\C+(n-1)B&C&B&\dots&B\\\vdots&\vdots&\vdots&\vdots&\vdots\\C+(n-1)B&B&B&\dots&C
		\end{bmatrix}\\
		&=\det(C+(n-1)B)\det\begin{bmatrix} I&B&B&\dots&B\\I&C&B&\dots&B\\\vdots&\vdots&\vdots&\vdots&\vdots\\I&B&B&\dots&C
		\end{bmatrix}
	\end{align*}
	Further, by  column transformations $C_i\longrightarrow C_i-BC_1$ we get that
	\begin{align*}
		\det(A)
		&=\det(C+(n-1)B)\det
		\begin{bmatrix} I&O&O&\dots&O\\I&C-B&O&\dots&O\\
			\vdots&\vdots&\vdots&\vdots&\vdots
			\\I&O&O&\dots&C-B
		\end{bmatrix}\\
		&=\det(C+(n-1)B)\det(C-B)^{n-1}.            
	\end{align*}        
\end{proof}

In the following lemma, we show that the graph $H_1$ has integral spectrum.
\begin{lemma}
	The adjacency spectrum of the graph $H_1$ is $$\left\{(n-1)^{(1)},(n-3)^{(1)}, (1)^{(2n)}, (-1)^{(2n)},  (-n+3)^{(1)},  (-n+1)^{(1)}\right\}.$$
\end{lemma}
\begin{proof}
	By Lemma \ref{lem1.5}, we have
	\begin{align*}
		det(xI-K)=&det(xI-D+(n-1)C)det(xI-D+C)^{n-1}\\
		=&(x-1)^{2n}(x+1)^{2n} \left(x+ n-3 \right)  \left( x-n+3 \right)  \left( x+n-1\right)\left( x-n+1 \right).        
	\end{align*}
	Hence the adjacency spectrum of the graph $H_1$ is\\ $\left\{(n-1)^{(1)},(n-3)^{(1)}, (1)^{(2n)}, (-1)^{(2n)},  (-n+3)^{(1)},  (-n+1)^{(1)}\right\}$.
\end{proof} 

\begin{corollary}
	$H_1$ is a bipartite graph.
\end{corollary}
\begin{proof}
Observe that the spectrum of $H_1$ is symmetric about origin. From the well known result \cite{dbwest} we get  proof of the corollary.	
\end{proof}
Now we consider the induced subgraph $H_2$ of the graph $H$ with the vertex set $\displaystyle \bigcup_{j=0}^{n}S_j\setminus\left\{ N_{a_j} \colon j=1,2,...,n\right\}.$\\
Let
$$A=\begin{bmatrix}1&0&1&1\\0&1&1&1\\1&1&0&1\\1&1&1&0
\end{bmatrix},~B=\begin{bmatrix}1&0&0&1\\0&1&1&0\\0&1&0&1\\1&0&1&0 \end{bmatrix}.$$ 
The adjacency matrix of $H_2$ is 
$$\begin{bmatrix} A&B&B&\dots&B\\B^t&D&C&\dots&C \\B^t&C&D&\dots&C\\\vdots&\vdots&\vdots&\vdots&\vdots\\B^t&C&C&\dots&D \end{bmatrix}.$$
Note that $H_2$ is subgraph of $H$ and $H_1$ is subgraph of $H_2$ and $H_2$ contains new idempotents  from the set $S_0$. \\
In the following lemma, we find the spectrum of $A(H_2)$. Observe that the spectrum of $A(H_2)$ is not integral but symmetric about origin. Hence  $H_2$ is bipartite graph.
\begin{lemma} \label{lem3.9}
	Spectrum of $A(H_2)$ is  \\
	$  \left\{(n)^{(1)}, (-n)^{(1)} ,(1)^{(2n-3)}, (-1)^{(2n-2)},   \left(\frac{n+3+\sqrt{n^2+10n-7}}{2}\right)^{(1)}, \left(\frac{n+3-\sqrt{n^2+10n-7}}{2}\right)^{(1)}\right\}.$
\end{lemma}
\begin{proof}
	Applying column operations $C_3-C_2, C_4-C_2,\cdots, C_n-C_2$ on $xI-A(H_2)$, we get
	$$\begin{bmatrix}
		xI-A& -B&O&O&\dots&O\\
		-B^t&xI-D&-p(x)&-p(x)&\dots&-p(x)\\
		-B^t&-C&p(x)&O&\dots&O\\
		-B^t&-C&O&p(x)&\dots&O\\
		\vdots&\vdots&\vdots&\vdots&\dots&\vdots\\
		-B^t&-C&O&O&\dots&p(x)\\
	\end{bmatrix},$$ where $p(x)=xI+C-D$.\\
	By applying the row operations $R_2+R_3, R_2+R_4, R_2+R_5,\cdots,R_2+R_n,$ we get\\
	$$\begin{bmatrix}
		xI-A& -B&O&O&\dots&O\\
		-(n-1)B^t&xI+(n-2)C-D&O&O&\dots&O\\
		B^t&-C&p(x)&O&\dots&O\\
		B^t&-C&O&p(x)&\dots&O\\
		\vdots&\vdots&\vdots&\vdots&\dots&\vdots\\
		B^t&-C&O&O&\dots&p(x)\\
	\end{bmatrix},$$ where $p(x)=xI+C-D$.\\
	Therefore 
	\begin{align*}
		\det(xI-A(H_2))=&(\det(xI+C-D))^{n-2}\det\begin{bmatrix}
			xI-A&B\\(n-1)B^t&xI-(n-2)C-D\end{bmatrix}\\
		=&(x-1)^{2n-4}(x+1)^{2n-3}(x-1)(x+1)(x-n)(x+n)^2\\
		&(x^2-(n+3)x-(n-4)).\\
		=&(x-1)^{2n-3}(x+1)^{2n-2}(x-n)(x+n)^2(x^2-(n+3)x-(n-4)).
	\end{align*}
	Hence the spectrum of $A(H_2)$ is  \\
	$  \left\{(n)^{(1)}, (-n)^{(1)} ,(1)^{(2n-3)}, (-1)^{(2n-2)},   \left(\frac{n+3+\sqrt{n^2+10n-7}}{2}\right)^{(1)}, \left(\frac{n+3-\sqrt{n^2+10n-7}}{2}\right)^{(1)}\right\}$.
\end{proof}
Now we consider the induced subgraph  $H_3$ of the graph $H$ with the vertex set $\displaystyle \bigcup_{j=0}^{n}S_j.$
\\\noindent Let
$$L=\begin{bmatrix}1&0&0&1&0\\0&1&1&0&0\\0&1&0&1&0\\1&0&1&0&0 \end{bmatrix},~N=\begin{bmatrix}0&0&0&1&0\\0&0&1&0&0\\0&1&0&0&0\\1&0&0&0&0\\0&0&0&0&0\end{bmatrix},~M=\begin{bmatrix}0&0&1&1&1\\0&0&1&1&1\\1&1&0&0&1\\1&1&0&0&1\\1&1&1&1&1\end{bmatrix}. $$

The adjacency matrix of $H_3$ is
$$
\begin{bmatrix} A&L&L&\dots&L\\L^t&M&N&\dots&N \\L^t&N&M&\dots&N\\\vdots&\vdots&\vdots&\vdots&\vdots\\L^t&N&N&\dots&M \end{bmatrix}.$$
Note that in the graph $H_3$, there are nilpotent elements as new vertices  which are not in $H_2$.\\

In the following lemma, we find the spectrum of $A(H_3)$. Observe that spectrum of $H_3$ is not integral but symmetric about origin.
\begin{lemma}\label{lem2.12}
	Spectrum of $A(H_3)$ is \\
	$\ \left\{ (n)^{(1)}, (-n)^{(2)}, (3)^{(n-2)},(1)^{(n-1)},(-1)^{(3n-2)},\left(\frac{n+5+\sqrt{n^2+6n-7}}{2}\right)^{(1)},\left(\frac{n+5-\sqrt{n^2+6n-7}}{2}\right)^{(1)}\right\}$.
\end{lemma}
\begin{proof}
	As in the proof  of Lemma \ref{lem3.9}, we have
		\begin{align*}
		\det(xI-A(H_3))=&(\det(xI+N-M))^{n-2}\det\begin{bmatrix}
			xI-A&-L\\-(n-1)L^t&xI-(n-2)N-M\end{bmatrix}\\
		=&(x-3)^{n-2}(x-1)^{n-2}(x+1)^{3n-5}(x-1)(x+1)^3(x+n)^2(x-n)\\
		&(x^2-(n+5)x+(n+8))\\
		=&(x-3)^{n-2}(x-1)^{n-1}(x+1)^{3n-2}(x+n)^2(x-n)\\&(x^2-(n+5)x+(n+8)).
	\end{align*}
	Hence the spectrum of $A(H_3)$ is \\
	$\left\{ (n)^{(1)}, (-n)^{(2)}, (3)^{(n-2)},(1)^{(n-1)},(-1)^{(3n-2)},\left(\frac{n+5+\sqrt{n^2+6n-7}}{2}\right)^{(1)},\left(\frac{n+5-\sqrt{n^2+6n-7}}{2}\right)^{(1)}\right\}$.
\end{proof}
\noindent Now let us consider of the induced subgraph $H_4$ of the graph $H$ with the vertex set 
\begin{align*}
	\left\{ E_{a_i,a_j} \colon a_i\neq 0,1, a_j\neq 1~\text{in}~F\right\}
	=\bigcup_{j=2}^{n}\left\{E_{\frac{a_j}{a_j-a_i},a_j} \colon i\ne 1,j\right\}.
\end{align*}
Let   $V_{jk}=[v_{rs}]_{(n-1)\times (n-1)},$ for $1\leq j\leq k\leq n-1$, be a matrix with
$$v_{rs}=\begin{cases}
	1, ~~&\text{if}~~r=j~\text{or}~s=k\\
	0,~~&\text{otherwise}.
\end{cases} $$
Let $$L=\begin{bmatrix}
	O&V_{11}&V_{12}&\dots &V_{1,n-1}&V_{1,n-1}\\
	O&O&V_{22}&\dots &V_{2,n-1}&V_{2,n-1}\\
	O&O&O&\dots &V_{2,n-1}&V_{2,n-1}\\
	\vdots&\vdots&\vdots&\vdots\\
	O&O&\dots &\dots&O&V_{n-1,n-1}\\
	O&O&\dots &\dots&O&O\\
\end{bmatrix}.$$
Then $A(H_4)=L+L^t.$\\

 In the following lemma, we find the spectrum of $A(H_4)$. Observe that it is an integral spectrum.

\begin{lemma}
	The spectrum  of $A(H_4)$ is $$\left\{(2n-3)^{(1)},(n-3)^{(n-1)}, (-(n-1))^{(n-1)},(-1)^{\left(\frac{n(n-3)}{2}\right)},(1)^{\left(\frac{n(n-3)}{2}+1\right)}\right\}.$$
\end{lemma}
\begin{proof}
	Since $H_4$ is $(2n-3)$-regular graph, $2n-3$ is its largest eigenvalue with multiplicity 1 and vector of all $1's$ is corresponding eigenvector.
	Reduced row echelon form of $A(H_4)+(n-1)I$ is 
	$$\begin{bmatrix}I_{(n)(n-1)-(n-1)}&B\\O_{(n-1)\times((n-1)n-(n-1)) }&O_{(n-1)} \end{bmatrix},$$ where
	$B=\begin{bmatrix}A\\AE_{1,1}\\ \vdots\\ AE_{1,n-1} \end{bmatrix}$
	and $A=\begin{bmatrix}1&-1&0&...&0\\1&0&-1&...&0\\\vdots&\vdots&\vdots&...&\vdots\\1&0&0&...&-1\\1&0&0&...&0 \end{bmatrix}_{(n-1)\times(n-1)}.$\\Therefore
	$-(n-1)$ is an eigenvalue  with multiplicity $n-1$.
	Reduced row echelon form of  $A(H_4)-(n-3)I$ is 
	$$\begin{bmatrix}I_{(n)(n-1)-(n-1)}&C\\O_{(n-1)\times((n-1)n-(n-1)) }&O_{(n-1)} \end{bmatrix},$$ where
	$C=\begin{bmatrix}D\\DE_{1,2}\\ \vdots\\DE_{1,n-1}\end{bmatrix}$ and $D=\begin{bmatrix}-a&-a&b&b&\dots&b\\-a&b&-a&b&\dots&b\\-a&b&b&-a&\dots&b\\\vdots&\vdots&\vdots&\vdots&\dots&\vdots\\
		-a&b&b&b&\dots&-a\\-1&0&0&0&\dots&0
	\end{bmatrix}_{(n-1)\times (n-1)}$ \\and $2a+(n-2)b=-1$.
	Therefore
	$n-3$ is an eigenvalue  with multiplicity $n-1$.\\
	An eigenvector $[x_1,x_2,\cdots,x_{(n-1)n)}]^t$ of $A(H_4)$ associated to the eigenvalue 1  satisfies $$x_{j(n-1)+k}=-x_{k(n-1)+j+1}~~\text{ for all}~~j=0,1,\cdots,n-1; k=1,2,\cdots,n-3.$$
	Also,\\ $x_{2(n-1)+2},x_{3(n-1)+2},\cdots,x_{(n-1)(n-1)+2};x_{3(n-1)+3},\cdots,x_{(n-1)(n-1)+3};\cdots;x_{(n-1)(n-1)+(n-1)}$ are parameters.
	Therefore there are $\frac{n(n-3)}{2}+1$ parameters. Hence nullity of $A(H_4)-I$ is $\frac{n(n-3)}{2}+1$. Therefore $1$ is an eigenvalue of $A(H_4)$ with multiplicity $\frac{n(n-3)}{2}+1$. \\
	An eigenvector $[x_1,x_2,\cdots,x_{(n-1)n)}]^t$ of $A(H_4)$   associated to the eigenvalue $-1$ satisfies $$x_{j(n-1)+k}=x_{k(n-1)+j+1}~~\text{for all} ~~j=0,1,\cdots,n-1; k=1,2,\cdots,n-3.$$
	Also, $x_{3(n-1)+2},\cdots,x_{(n-1)(n-1)+2};x_{3(n-1)+3},\cdots,x_{(n-1)(n-1)+3};\cdots;x_{(n-1)(n-1)+(n-1)}$ are parameters.
	Therefore there are $\frac{n(n-3)}{2}$ parameters. Hence nullity of $A(H_4)+I$ is $\frac{n(n-3)}{2}$. Therefore $-1$ is an eigenvalue of $A(H_4)$ with multiplicity $\frac{n(n-3)}{2}.$ 
\end{proof}
Let $K$ be an induced subgraph of $H$ on the vertex set  $S_j$ then the adjacency matrix of $K$ is given by$$C=\begin{bmatrix}0&0&1&1&1&1_{1,n-1}\\0&0&1&1&1&1_{1,n-1}\\1&1&0&0&1&0_{1,n-1}\\1&1&0&0&1&0_{1,n-1}\\1&1&1&1&1&1_{1,n-1}\\1_{n-1,1}&1_{n-1,1}&0_{n-1,1}&0_{n-1,1} &1_{n-1,1}&0_{n-1,n-1}\end{bmatrix}~~~\text{for all}~~i=1,2,...,n.$$ Let $C_{jk}$ denotes the matrix which gives the adjacency between $S_j$ and $S_k$. \\
Therefore $$C_{jk}=\begin{bmatrix}0&0&0&1&0&0_{1,n-1}\\0&0&1&0&0&0_{1,n-1}\\0&1&0&0&0&e^{(j)}_{1,n-1}\\1&0&0&0&0&e^{(j)}_{1,n-1}\\0&0&0&0&0&e^{(j)}_{1,n-1}\\0_{n-1,1}&0_{n-1,1}&(e^{(k)}_{1,n-1})^t&(e^{(k)}_{1,n-1})^t&(e^{(k)}_{1,n-1})^t&V_{jk}\end{bmatrix}~~~\text{for all}~j,k=1,2,\cdots,n.$$
Let $B_0$ denote the matrix which gives the adjacency between $S_0$ and $S_j$.\\ Therefore
$$B_0=\begin{bmatrix}1&0&0&1&0&0_{1,n-1}\\0&1&1&0&0&0_{1,n-1}\\0&1&0&1&0&0_{1,n-1}\\1&0&1&0&0&0_{1,n-1} \end{bmatrix}.$$
Let $A(H)$ be the adjacency matrix of $H.$\\ Hence
$$A(H)=\begin{bmatrix}
	A& B_0&B_0&B_0&\dots&B_0\\
	B_0^t&C&C_{11}&C_{12}&\dots&C_{1,n-1}\\
	B_0^t&C_{11}&C&C_{22}&\dots&C_{2,n-1}\\
	B_0^t&C_{21}&C_{22}&C&\dots&C_{3,n-1}\\
	\vdots&\vdots&\vdots&\vdots&\dots&\vdots\\
	B_0^t&C_{n-1,1}&C_{n-1,2}&C_{n-1,3}&\dots&C\\
\end{bmatrix}.$$
\par Since the graph $H$  is $(2n+3)$-regular, $2n+3 $ is its eigenvalue with multiplicity one. 
By relabeling the vertices of $H$, we  can write
\begin{equation}
	H=H_3\sqcup H_4.\label{eq3.1}
\end{equation}
Let $A(H_3,H_4)$ be the matrix representing adjacency between the graphs $H_3$ and $H_4$. Then
$$A(H)=\begin{bmatrix}A(H_3)&A(H_3,H_4)\\
	A(H_4,H_3)&A(H_4)\end{bmatrix}. 
$$ 
Observe that  $A(H_3,H_4)=\begin{bmatrix}O&X_1&X_2&...&X_n\end{bmatrix}^t$
with $O$ is a zero matrix and \\
$X_i=\begin{bmatrix}0_{1,n-1}&...&0_{1,n-1}&\overbrace{1_{1,n-1}}^{i^{th} ~column}&0_{1,n-1}&...&0_{1,n-1}\\
	0_{1,n-1}&...&0_{1,n-1}&1_{1,n-1}&0_{1,n-1}&...&0_{1,n-1}\\
	e^{(i-1)}_{1,n-1}&...&e^{(i-1)}_{1,n-1}&0_{1,n-1}&e^{(i)}_{1,n-1}
	&...&e^{(i)}_{1,n-1}\\
	e^{(i-1)}_{1,n-1}&...&e^{(i-1)}_{1,n-1}&0_{1,n-1}&e^{(i)}_{1,n-1}
	&...&e^{(i)}_{1,n-1}\\
	e^{(i-1)}_{1,n-1}&...&e^{(i-1)}_{1,n-1}&1_{1,n-1}&e^{(i)}_{1,n-1}
	&...&e^{(i)}_{1,n-1}
\end{bmatrix}.$\\

 It is clear that $A(H_4,H_3)=\left(A(H_3, H_4)\right)^t.$\\

In the following lemma, we find the spectrum of $A(H)$. Observe that it is an integral spectrum and except largest eigenvalue the spectrum of $A(H)$ is symmetric about the origin.
\begin{lemma}
	The spectrum of $A(H)$ is\\
	$\left\{(2n+3)^{(1)},(n+1)^{(n+1)},(-(n+1))^{(n+1)},(1)^{\left(\frac{n(n+1)}{2}\right)},(-1)^{\left(\frac{(n+1)(n+2)}{2}\right)} \right\}.
	$ 
\end{lemma}
\begin{proof}
	Since the graph $H$ is $(2n+3)$-regular, $A(H)$ has the largest eigenvalue $2n+3$ with multiplicity 1. \\
	Let $[x_1,x_2,\cdots,x_{(n+2)^2}]^t$ be an eigenvector associated to an eigenvalue $n+1$.\\Then  $x_{(n+2)^2-n},\cdots,x_{(n+2)^2-1},x_{(n+2)^2}$  are free variables and other variables are given by 
	\begin{align*}
		&x_{(n+2)^2-(n+4)-k}=x_{(n+2)^2-k}~~\text{for all} ~~k=0,1,2,\cdots,n-1~~\text{and}~~ x_{(n+2)^2-n-5}=x_{(n+2)^2-n};\\
		&x_1=2x_{(n+2)^2-n}-x_{(n+2)^2-n+1},~ x_{(n+4)k}=2x_{(n+4)n+k}-x_{(n+2)^2-n+1},\\
		&\text{for all}~~k=1,2,\cdots, n-1;\\
		&x_{(n+2)^2-n-2}=x_{(n+2)^2-n-1}=-\displaystyle \sum_{k=0}^{n}x_{(n+2)^2-k}+\frac{n}{2} x_{(n+2)^2-n+1};\\
		&x_{(n+2)^2-n-2-(n+4)k}=x_{(n+2)^2-n-1-(n+4)k}\\&=-\displaystyle \sum_{k=0}^{n}x_{(n+2)^2-k}+x_{(n+2)^2-k+1}+\frac{n-2}{2} x_{(n+2)^2-n+1} ~~~\text{for all}~~k=1,2,3,\cdots,n;\\
		&x_{5+(n+4)k}=x_{(n+2)^2-n}+x_{(n+2)^2-n+2+k}-x_{(n+2)^2-n+1}~~\text{for all} ~~k=0,1,2,\cdots,n-1;\\
		&x_{8+(n+4)k}=x_{(n+2)^2-n}+x_{(n+2)^2-n+2+k}-x_{(n+2)^2-n+1};\\
		&x_{10+(n+4)k}=x_{(n+2)^2-n+2}+x_{(n+2)^2-n+3+k}-x_{(n+2)^2-n+1}~~\text{for all}~~k=0,1,2,\cdots,n-2.
	\end{align*} 
	Therefore $n+1$ is an eigenvalue of $A(H)$ with multiplicity $n+1$.\\
	Let $[y_1,y_2,\cdots,y_{(n+2)^2}]^t$ be an eigenvector associated to an eigenvalue $-(n+1)$.\\Then  $y_{(n+2)^2-n},\cdots,y_{(n+2)^2-1},y_{(n+2)^2}$  are free variables and other variables are given by \\
	$y_1=y_2=y_{9+(n+4)k}=0~~\text{for all}~~k=0,1,2,\cdots,n-1;\\
	y_{9+(n-1)+(n+4)k}=-y_{(n+2)^2-n-1+k}~~\text{for all}~~k=0,1,\cdots,n-1;\\
	y_3=-y_4=x_{(n+2)^2-n-1}-x_{(n+2)^2-n};\\
	y_5=-y_8=x_{(n+2)^2-n+2}-x_{(n+2)^2-n}, y_6=-y_7=x_{(n+2)^2-n+2}-x_{(n+2)^2-n-1};\\
	y_{10+j+(n+4)k} =y_{(n+2)^2-n+2+j}-y_{(n+2)^2-n+2+k}~~\text{for all}~~k=0,1,2,3,\cdots,n-2,\\
	\text{for all}~~j=1,2,\cdots,n-2;\\
	y_{10+(n-1)+j+(n+4)k}=y_{(n+2)^2-n+k}-y_{(n+2)^2-n+2+j},~~\text{for all}~~k=0,1,2,\cdots,n-2;\\~~\text{for all}~~j=1,2;\\
	y_{10+(n-1)+j+(n+4)k}=-y_{(n+2)^2-n+k}+y_{(n+2)^2-n+2+j},~~\text{for all}~~k=0,1,2,\cdots,n-2;\\~~\text{for all}~~j=3,4$.\\
	Therefore $-(n+1)$ is an eigenvalue of $A(H)$ with multiplicity $n+1$.\\
	Let $[z_1,z_2,\cdots,z_{(n+2)^2}]^t$ be an eigenvector associated to an eigenvalue $1$.\\Then,\\
	$\left\{ z_{8+(n+4)k+j} \colon k=0,1,2,\cdots,n+1, j=0,1,2,\cdots, k+1\right\}$$\setminus \left\{ z_{9+(n+4)k}\colon k=0,1,2,\cdots,n \right\}$ \\ is a set of free variables and other variables are given by \\
	$z_{9+(n+4)k}=0,~~\text{for all}~~k=1,2,\cdots,n;\\
	z_{5}=-z_8, ~y_{10+(n+1)k+j}=-z_{10+(n+1)k+(n+4)j},~~\text{for all}~~ j=0,1,\cdots,n;\\~\text{for all}~~k=0,1,\cdots,n-1;\\
	z_3=-z_4=-\displaystyle\sum_{k=0}^{n-1}x_{8+(n+4)k};~z_6=-z_7=z_8-\displaystyle\sum_{k=0}^{n-1}x_{10+(n+4)k};\\
	z_{6+(n+4)j}=-z_{7+(n+4)j}=\sum_{i=1}^{j+1}z_{8+(n+4)j+i}-\displaystyle\sum_{k=j+2}^{n-1}x_{10+(n+4)k+j},\\~~\text{for all}~j=0,1,2,\cdots ,n-1$.\\
	Therefore $1$ is an eigenvalue of $A(H)$ with multiplicity $1+2+\cdots +n=\frac{n(n+1)}{2}.$\\
	Let $[u_1,u_2,\cdots,u_{(n+2)^2}]^t$ be an eigenvector associated to an eigenvalue $1$.\\Then  $\left\{ u_{8+(n+4)k+j} \colon k=0,1,2,\cdots,n+1,j=0,1,2,\cdots, k+1\right\}\cup\left\{u_4\right\}$  is a set of free variables and the relation between other variables are given by \\
	$
	u_1=-u_4-u_8-\displaystyle\sum_{k=1}^{n-1} u_{8+(n+4)k};\\
	u_2+u_4=\displaystyle\sum_{k=0}^{n}u_{8+(n+4)k}+u_{9+(n+4)k}+2\sum_{k=0}^{n-1}u_{9+(n+5)+(n+4)k}\\+2\sum_{k=1}^{n-1}u_{9+(n+6)+(n+4)k}+2\sum_{k=2}^{n-1}u_{9+(n+7)+(n+4)k}+\cdots+2u_{(n+2)^2};\\
	u_3=u_4,u_5=u_8,\\u_{6+(n+4)i}=u_{7+(n+4)i}=-u_{8+(n+4)i}-u_{9+(n+4)i}-\sum_{k=1}^{n-1}u_{10+(n+4)k+(n+4)i},\\~~\text{for all}~~i=0,1,2,\cdots, n-1;\\
	u_{10+j}=u_{10+(n+4)(j+1)}~~\text{for all}~j=0,1,2,\cdots,n-1 $.\\
	Therefore $-1$ is an eigenvalue of $A(H)$ with multiplicity \\ $1+2+\cdots +n+(n+1)=\frac{(n+1)(n+2)}{2}.$
\end{proof}
In the following example, we find the characteristic polynomial of $A(H)$ for the ring $M_2(Z_2)$.
\begin{example}
	Let $R=M_2(\mathbb Z_2)$.
	\begin{center}
	\begin{figure}[h]
	\hspace{1.5in}	\includegraphics[width=2in]{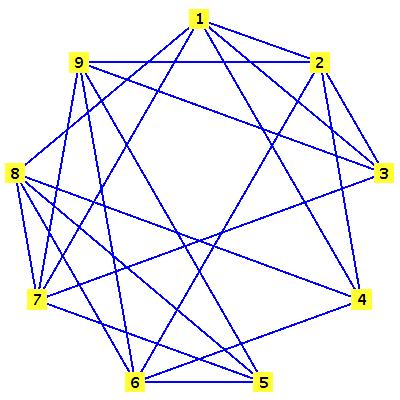}
		\caption{$\Gamma(M_2(Z_2))$}
	\end{figure}
\end{center}
	The adjacency matrix of $\Gamma(R)$ is 
	$A(H)=\left[ \begin {array}{ccccccccc} 0&1&1&1&0&0&1&1&0
	\\ \noalign{\medskip}1&0&1&1&0&1&0&0&1\\ \noalign{\medskip}1&1&0&0&0&0
	&1&0&1\\ \noalign{\medskip}1&1&0&0&0&1&0&1&0\\ \noalign{\medskip}0&0&0
	&0&0&1&1&1&1\\ \noalign{\medskip}0&1&0&1&1&0&0&1&1
	\\ \noalign{\medskip}1&0&1&0&1&0&0&1&1\\ \noalign{\medskip}1&0&0&1&1&1
	&1&0&0\\ \noalign{\medskip}0&1&1&0&1&1&1&0&0\end {array} \right] $\\
	Note that the characteristic polynomial of $A(H)$ is \\
	$\left( x-1 \right)  \left( {x}^{2}-3\,x-8 \right)  \left( x+2
	\right) ^{2} \left( {x}^{2}-2 \right) ^{2}$.
\end{example}

\section{Bounds on eigenvalues of adjacency matrix of  $\Gamma(M_2(F))$}

In this section, the relation $\sim$ is used to express the graph $\Gamma(M_2(F))$ as the generalized join of graphs on equivalence classes. 
The generalized join of a family of graphs is defined as below, which is used to find  the adjacency spectrum $\sigma_A(G)$ and Laplacian spectrum $\sigma_L(G)$ of a graph $G$.
\begin{definition}
	Let $H=(I, E)$ be a graph with a vertex set $I=\displaystyle \left\{1,2,3,\cdots, n\right\}$ and edge set $E$. 
	Let $\displaystyle \mathcal{F}=\{G_i=(V_i, E_i)\colon i\in I\}$  be a family of graphs such that $V_i\cap V_j=\phi$ for all $i\neq j$.  The $H$-generalized join of the family $\mathcal{F}$ is denoted by  $\displaystyle \bigvee_{H}\mathcal{F}$ and is a graph obtained by replacing each vertex $i$ of $H$ by the graph $G_i$ and joining each vertex of $G_i$ to every  vertex of $G_j$ whenever $i$ and $j$ are adjacent in $H$.
\end{definition} 

\noindent 
Recall the following notations:\\
$S_0=\left\{ M, N,E_0, E^0\right\},\\
S_{j}=\displaystyle \left\{ E_{\frac{-1}{a_j}},E^{-a_j}, F^{a_j},F_{\frac{1}{a_j}},N_{a_j}\right\},\\
T_j=\left\{E_{\frac{a_j}{a_j-a_i},a_j}~\colon~ i\neq 0,~\text{and for all}~ j=1,2,3...,n\right\}$.\\

In the following proposition, we express $\Gamma(M_2(F))$ as a generalized join of a family of complete and null graphs.
\begin{proposition}\label{proposition2.66}
	
	$$\Gamma(M_2(F))=\bigvee_H \{\Gamma([x])~|~x\in\mathcal{F}\},$$  
	where $\mathcal{F}=S_0\cup S_1\cup \dots \cup S_n\cup T_0\cup T_1\cup \cdots \cup T_n. $ 
	Moreover the induced subgraph on $[x]$ of  $~\Gamma(M_2(F))$  is a null graph  if $x$ is an idempotent and is the complete graph if $x$ is a nilpotent.
\end{proposition}
\begin{proof}
	
	By the Lemma \ref{lem2.12},  equivalence classes of the relation $\sim$ on $M_2(F)$ are $\displaystyle \mathcal{E}=\left\{ [x] ~|~ x\in\mathcal{F}\right\}.$\\
	Let $H$ be a graph with a vertex set $\displaystyle \mathcal{E}$ and any two vertices $[x]$, and $[y]$ are adjacent if $x$ and $y$ are adjacent in $\Gamma(M_2(F))$. 
	Hence  $\Gamma(M_2(F))=\displaystyle\bigvee_H \{\Gamma([x])~|~x\in\mathcal{F}\}.$\\ 
	Let $x$ be a nonzero idempotent in $M_2(F)$  and $y\in \Gamma([e]).$ \\
	Therefore $y=u_1x=xv_1,~z=u_2x=xv_2.$ Hence $yz=u_1x^2v_2=u_1xv_2\neq 0$. That is any two vertices $x,y \in  \Gamma([x])$ are not adjacent. Hence $\Gamma([x])$ is a null graph.\\
	Let $x$ be a nonzero nilpotent element in $M_2(F)$  and $y,z\in \Gamma([z]).$ \\
	Therefore $y=u_1x=xv_1,~z=u_2x=xv_2.$ This gives $xy=u_1x^2v_2=u_10v_2=0$. Therefore any two vertices $y,z$ in $\Gamma([x])$ are  adjacent. Hence $\Gamma([x])$ is a complete graph. 
\end{proof}
\noindent Recall the following result due to Domingos M. Cardoso et.al   \cite{cardoso2013spectra}. In this result the spectrum of a generalized join graph is expressed in terms of the spectrum of each component graph.
\begin{proposition}\label{proposition2.5}(\cite{cardoso2013spectra}, Theorem 5)
	Let $K$ be a graph on set $I=\left\{1,2,\cdots ,n\right\}$
	and let $\displaystyle\mathcal{F}=\{ G_i \colon i\in I \}$ be a family of    
	$n$ pairwise disjoint $r_i$-regular graphs $G_i$ of order $m_i$ respectively.\\
	Let \begin{align*}\displaystyle& N_i=\begin{cases}\displaystyle \sum_{j\in N(i)}m_j,~~& N(i)\neq \phi,\\
			0,~~~~& \text{otherwise}  \end{cases}
		 \end{align*}
		 $ P=diag(r_1,r_2,\cdots,r_n),~~Q=diag(N_1, N_2,\cdots,N_n),~\text{and}~ R=diag(m_1,m_2,\cdots,m_n)$.\\
	If  $\displaystyle G=\bigvee_{K}\mathcal{F}$,  then 
	\begin{equation}
		\sigma_{A}(G)=\left(\bigcup_{i}^{n}\left(\sigma_A(G_i)\diagdown\left\{r_i\right\} \right)\right) \bigcup\sigma(P+\sqrt{R}A(H)\sqrt{R}) ~\text{and}
		\end{equation}
		\begin{equation}
		\sigma_{L}(G)=\left(\bigcup_{i}^{n}\left(N_i+(\sigma_L(G_i)\diagdown\left\{0\right\})\right)\right) \bigcup\sigma(Q+\sqrt{R}A(H)\sqrt{R}). 
	\end{equation}
\end{proposition}

In the following corollary, we express the spectrum of $\Gamma(M_2(F))$ in terms of the spectrum of $T+A(H)$.
\begin{corollary}
	Let $F$ be a field and $n=|F|-1$. Then\\
	$
	\sigma_A(\Gamma(M_2(F)))=\left\{0^{(n+1)(n+2)(n-1)}, -1^{(n+2)(n-1)}\right\} \bigcup  \sigma\left(T+A(H)\right)$, \\
	where
	$T=diag\left(\underbrace{0}_{4 ~times},\underbrace{Y}_{n~ times}\right)~\text{with}~Y=diag(\underbrace{0}_{4~times},1,\underbrace{0}_{n-1~times}).$
\end{corollary} 
\begin{proof}
	By Proposition \ref{proposition2.5}, $ \Gamma(M_2(F))=\displaystyle\bigvee_H\{\Gamma([x])~|~x\in\mathcal{F}\}, $
	where 
	$H$ is the graph defined as in equation (\ref{eq3.1}).\\
	Clearly, $\Gamma([x])=\overline{K}_{n}$ if  $x$ is an idempotent element and $\Gamma([x])=K_{n}$ if $x$ is a nilpotent element. 
	Matrices $P, R$ in the Proposition \ref{proposition2.5} becomes 
	$  P=diag\left(\underbrace{0}_{4 ~times},\underbrace{X}_{n~ times}\right)$ with $X=diag(\underbrace{0}_{4~times},n,\underbrace{0}_{n-1~times})$ and $R=nI_{(n+2)^2}.$
	Therefore, by Proposition \ref{proposition2.5},
	\begin{align*}
		&\sigma_A(\Gamma(M_2(F)))&\\
		&=\left(\sigma_A(\Gamma([M]))\diagdown\{n-1\}\right)\bigcup\left(\sigma_A(\Gamma([N]))\diagdown\{n-1\}\right)\\&\bigcup \left(\sigma_A(\Gamma([E_0]))\diagdown\{0\}\right)
		\bigcup \left(\sigma_A(\Gamma([E^0]))\diagdown\{0\}\right)\\&\bigcup_{a} \left(\sigma_A(\Gamma([E_a]))\diagdown\{0\}\right) \bigcup_{a}\left(\sigma_A(\Gamma([E^{-1/a}]))\diagdown\{0\}\right)\\
		&\bigcup_{a}\left(\sigma_A(\Gamma([F_{-a}]))\diagdown\{0\}\right) \bigcup_{a} \left(\sigma_A(\Gamma([F^{-a}]))\diagdown\{0\}\right)\\& 
		\bigcup_{j,a}\left(\sigma_A(\Gamma([E_{j,-1/a}]))\diagdown\{0\}\right)  \bigcup_{a}\left(\sigma_A(\Gamma([N_{1/a}]))\diagdown\{n-1\}\right)\\
		& \bigcup\sigma\left(P+\sqrt{R}A(H)\sqrt{R}\right)
		\end{align*}
		\begin{align*}
		&=\left( \sigma_A\left(K_n\right)\setminus\{0\}\right)\bigcup\left( \sigma_A\left(K_n\right)\setminus\{0\}\right)
		\bigcup \left( \sigma_A\left(\overline{K}_n\right)\diagdown\{n-1\}\right)\\
		&\bigcup \left( \sigma_A\left(\overline{K}_n\right)\diagdown\{n-1\}\right)
		\bigcup_{a} \left( \sigma_A\left(\overline{K}_n\right)\diagdown\{n-1\}\right)\bigcup_{a} \left( \sigma_A\left(\overline{K}_n\right)\diagdown\{0\}\right)
		\end{align*}
				\begin{align*}
		&\bigcup_{a}\left(\sigma_A\left(\overline{K}_n\right)\setminus\{0\}\right) \bigcup_{a} \left( \sigma_A(\overline{K}_n)\diagdown\{0\}\right)\\
		& \bigcup_{j,a} \left( (\sigma_A(\overline{K}_n)\diagdown\{0\})\right)\bigcup_{a} \left( \sigma_A(K_n)\diagdown\{0\}\right)\bigcup\sigma(P+\sqrt{R}A(H)\sqrt{R})
		\}\\
		&=\left\{(0)^{((n+1)(n+2)(n-1))}, (-1)^{((n+2)(n-1))}\right\} \bigcup  n\sigma\left(diag\left(T+A(H)\right)\right),
	\end{align*}
	where $T=diag\left(\underbrace{0}_{4 ~times},\underbrace{Y}_{n~ times}\right)$ with $Y=diag(\underbrace{0}_{4~times},1,\underbrace{0}_{n-1~times}).$
\end{proof}  
It is difficult to find eigenvalues of $T+A(H)$ even though eigenvalues of $A(H)$ are known.
We find upper and lower bounds for eigenvalues of the matrix $T+A(H)$ using the following inequality due to Weyl.
\begin{proposition}\label{pr3.6}(Weyl's inequality \cite{weyl})
	Let $A, B\in \mathbb{C}^{n\times n}$  be Hermitian matrices and $i,j, k, h\in \mathbb{N}$  with $j+k-1\leq i\leq l+h-n-1$ then 
	$\lambda_l(A)+\lambda_h(B)\leq \lambda_i(A+B)\leq \lambda_j(A)+\lambda_k(B),$
	where $\lambda_1(C)\geq \lambda_2(C)\geq ...\geq \lambda_n(C)$ are eigenvalues of any Hermitian matrix C.
\end{proposition}
We found the spectrum of $A(H)$ explicitly. If we change diagonal entries of $A(H)$ by adding diagonal matrix $T$ then it is difficult to find eigenvalues of $T+A(H)$. But using Weyl's inequality, we find upper and lower bounds for eigenvalues of $T+A(H)$ in the following proposition.
\begin{proposition}
	Let $\sigma(T+A(H))=\left\{\alpha_1\geq \alpha_2\geq ...\geq \alpha_{(n+2)^2}\right\}$.
	Then
	\begin{enumerate}
		\item $n+1\leq \alpha_1\leq 2n+4.$
		\item $n+1\leq \alpha_{i}\leq n+2~~\text{for all}~i=2,3,\cdots,n+1.$
		\item $1\leq \alpha_{n+2}\leq n+1.$
		\item $1\leq \alpha_i\leq 2~~\text{for all}~~i=n+3, \cdots,2n+2.$
		\item $\alpha_i=1~~\text{for all}~~i=2n+3,\cdots,n+1+\frac{n(n+1)}{2}.$
		\item $-1\leq \displaystyle{\alpha_{n+2+\frac{n(n+1)}{2}}}\leq 1.$
		\item $-1 \leq {\alpha_i} \leq 0~~\text{for all}~~i=n+3+\frac{n(n+1)}{2},\cdots,2n+3+\frac{n(n+1)}{2}.$
		\item $\alpha_i=-1~~\text{for all}~i=2n+4+\frac{n(n+1)}{2},\cdots,(n+2)^2-n-1.$
		\item $-(n+1)\leq \alpha_i\leq -n~~\text{for all}~~i=(n+2)^2-n,\cdots,(n+2)^2-1.$
		\item $ \alpha_{(n+2)^2}= -(n+1).$
	\end{enumerate} 
\end{proposition}
\begin{proof}
	Observe that eigenvalues of $T$ are\\
	$1=\lambda_1=\lambda_2=\cdots=\lambda_n>\lambda_{n+1}=0=\cdots=\lambda_{(n+2)^2}.$\\
	Also, eigenvalues of $A(H)$ are \\
	$(2n+3)=\mu_1$\\$>(n+1)=\mu_2=\cdots=\mu_n=\mu_{n+1}=\mu_{n+2}$\\
	$>1=\mu_{n+3}=\cdots=\mu_{n+2+\frac{n(n+1)}{2}}$\\
	$>-1=\mu_{n+3+\frac{n(n+1)}{2}}=\cdots=\mu_{(n+2)^2-n}$
	\\$>-(n+1)=\mu_{(n+2)^2-(n-1)}=\cdots=\mu_{(n+2)^2}.$\\
	Let $\alpha_1\geq \alpha_2\geq \cdots \geq \alpha_{(n+2)^2}$ be  eigenvalues of $T+A(H)$.\\
	Then by Proposition \ref{pr3.6}, we have  
	\begin{align*}
		\mu_l+\lambda_h\leq \alpha_i \leq \mu_{j}+\lambda_k ~~\text{for all}~ i, j, k~~\text{such that} ~j+k\leq i+1\leq l+h-(n+2)^2 
	\end{align*}
	Therefore we get $$\mu_2+\lambda_{(n+2)^2}\leq \alpha_1\leq \mu_1+\lambda_1,~~ie.,~~n+1\leq \alpha_1\leq 2n+4.$$ 
	For all  $i=2,3,\cdots,n+1$,  we get $$\mu_{n+2}+\lambda_{(n+2)^2}\leq \alpha_i\leq \mu_2+\lambda_1,~~ie.,~~n+1\leq \alpha_i\leq n+2 .$$
	Further we get $$1=\mu_{n+3}+\lambda_{(n+2)^2}\leq \alpha_{n+2}\leq \mu_2+\lambda_{n+1}=n+1. $$
	For $i=n+3,n+4,\cdots,2n+2$, we get
	$$1=\mu_{n+3}+\lambda_{(n+2)^2}\leq \alpha_i\leq \mu_{n+3}+\lambda_{1}=2.$$
	For $i=2n+3,2n+4,\cdots,n+1+\frac{n(n+1)}{2},$ we get
	$$1=\mu_{n+2+\frac{n(n+1)}{2}}+\lambda_{(n+2)^2}\leq \alpha_i\leq \mu_{n+3}+\lambda_{n+1}=1.$$
	Further we get
	$$-1=\mu_{n+3+\frac{n(n+1)}{2}}+\lambda_{(n+2)^2}\leq \alpha_{n+2+\frac{n(n+1)}{2}}\leq \lambda_1+\mu_{n+2+\frac{n(n+1)}{2}} =1.$$
	For $i=n+3+\frac{n(n+1)}{2},\cdots,2n+3+\frac{n(n+1)}{2},$ we get
	$$-1=\mu_{2n+4+\frac{n(n+1)}{2}}+\lambda_{(n+2)^2}\leq \alpha_i\leq \mu_{n+3+\frac{n(n+1)}{2}}+\lambda_{1}=0.$$
	
	For $i=2n+4+\frac{n(n+1)}{2},\cdots,(n+2)^2-n-1,$ we get
	$$-1=\mu_{(n+2)^2-n}+\lambda_{(n+2)^2}\leq \alpha_i\leq \mu_{n+3+\frac{n(n+1)}{2}}+\lambda_{n+1}=-1.$$
	For $i=(n+2)^2-n,\cdots,(n+2)^2-1,$ we get
	$$-n-1=\mu_{(n+2)^2}+\lambda_{(n+2)^2}\leq \alpha_{i}\leq \mu_{(n+2)^2-n+1}+\lambda_1=-n$$
	Finally we get
	$$-n-1=\mu_{(n+2)^2-n+1}+\lambda_{(n+2)^2}\leq \alpha_{(n+2)^2}\leq \mu_{(n+2)^2-n}+\lambda_{n+1}=-n-1$$
\end{proof}

\bibliographystyle{ieeetr}

\bibliography{ref}	
\par	{ Department of Mathematics, Abasaheb Garware College, Pune-411004, India.
\par \email{\emph{Email address:  krishnatmasalkar@gmail.com, ask.agc@mespune.in,\\ 
anita7783@gmail.com, lnk.agc@mespune.in  }}

\end{document}